\documentclass[a4paper,11pt]{amsart}

\usepackage{amsmath,amssymb,amsthm,amsfonts}
\usepackage{mathtools}
\usepackage{xcolor}
\usepackage{tikz}
\usetikzlibrary{shapes} 
\usetikzlibrary{decorations.pathmorphing} 
\usetikzlibrary{cd}
\usetikzlibrary{arrows}
\usepackage{ifthen}

\usepackage{ulem}
\usepackage{enumitem}
\usepackage{mdwlist} 
\usepackage{booktabs}
\usepackage{hyperref}
\hypersetup{
    colorlinks,
    linkcolor={red!50!black},
    citecolor={blue!50!black},
    urlcolor={blue!80!black}
}
\usepackage{calc}
\usepackage{aliascnt} 
\usepackage{amscd}
\usepackage{dutchcal} 

\newcommand{\cA}{\mathcal{A}}
\newcommand{\cB}{\mathcal{B}}

\newcommand{\cP}{\mathcal{P}}

\newcommand{\IC}{\mathbb{C}}
\newcommand{\IP}{\mathbb{P}}

\newcommand{\IZ}{\mathbb{Z}}
\newcommand{\IN}{\mathbb{N}}

\newcommand{\kk}{\mathbf{k}}

\newcommand{\blank}{\,\rule[0pt]{0.66em}{0.4pt}\,} 

\newcommand{\eqqcolon}{=\mathrel{\mathop:}}

\newcommand{\reg}{\mathcal O}
\newcommand{\std}{\mathrm{std}}

\DeclareMathOperator{\Ho}{\mathsf H}
\mathchardef\mhyphen="2D

\DeclareMathOperator{\Hom}{\mathsf{Hom}}

\DeclareMathOperator{\Aut}{Aut}

\DeclareMathOperator{\id}{\mathsf{id}}
\DeclareMathOperator{\pr}{pr}

\DeclareMathOperator{\supp}{supp}

\let\deg\relax
\DeclareMathOperator{\deg}{\mathsf{deg}}

\DeclareMathOperator{\Symm}{Sym}

\newcommand{\arXiv}[1]{\href{https://arxiv.org/abs/#1}{\texttt{arXiv:#1}}}

\newcommand{\hyref}[2]{\hyperref[#2]{#1~\ref*{#2}}}

\newcommand{\sym}{\mathfrak S}

\newtheorem{theorem}{Theorem}[section]
\newtheorem*{maintheorem}{Main Theorem}
\newtheorem*{theorem*}{Theorem}

  \newaliascnt{proposition}{theorem}
  \newtheorem{proposition}[proposition]{Proposition}
  \aliascntresetthe{proposition}
  
  \newaliascnt{lemma}{theorem}
  \newtheorem{lemma}[lemma]{Lemma}
  \aliascntresetthe{lemma}
  
  \newaliascnt{corollary}{theorem}
  
  \aliascntresetthe{corollary}

\theoremstyle{definition}
  \newaliascnt{definition}{theorem}
  
  \aliascntresetthe{definition}
  
  \newaliascnt{remark}{theorem}
  \newtheorem{remark}[remark]{Remark}
  \aliascntresetthe{remark}
  
  \newaliascnt{setup}{theorem}
  
  \aliascntresetthe{setup}
  
  \newaliascnt{question}{theorem}
  
  \aliascntresetthe{question}
  
  \newaliascnt{example}{theorem}
  
  \aliascntresetthe{example}
  
  \def\equationautorefname#1#2\null{(#2\null)}

\begin{document}

\title{Asymmetry of $\IP$-functors}
\author[A. Hochenegger]{Andreas Hochenegger}
\address{
Dipartimento di Matematica ``Francesco Brioschi'',
Politecnico di Milano,
via Bonardi 9,
20133 Milano,
Italy
}
\email{andreas.hochenegger@polimi.it}
\author[A. Krug]{Andreas Krug}
\address{
Institut f\"ur algebraische Geometrie,
Gottfried Wilhelm Leibniz Universit\"at Hannover,
Welfengarten 1,
30167 Hannover,
Germany
}
\email{krug@math.uni-hannover.de}


\maketitle

\begin{abstract}
Recently, a new definition of $\IP$-functors was proposed by Anno and Logvinenko.
In their article, the authors wonder whether this notion is symmetric in the sense that the adjoints of $\IP$-functors are again $\IP$-functors, the analogue being true for spherical functors.
We give geometric examples involving the Hilbert scheme of points on a surface that yield a negative answer.
\end{abstract}

\section{Introduction}

In order to understand the derived category $D(X)$ of a smooth, projective variety $X$,
the group of autoequivalences $\Aut(D(X))$ plays a crucial role.

There are the so-called standard autoequivalences $\Aut^\std(D(X))$, coming from the shift, pullback by an automorphism of $X$, and twisting with an invertible sheaf. Bondal and Orlov showed in \cite{BonOrl} that, if $\omega_X$ is ample or anti-ample, these are already all autoequivalences.

Things become more interesting if $\omega_X$ is neither ample or anti-ample. In that case, there can be autoequivalences that are not standard.
Seidel and Thomas showed in \cite{ST}, how to associate a spherical twist to a spherical object, which turned out to be a new kind of autoequivalence.
Their notion was generalised in several ways. On the one hand, Huybrechts in Thomas introduced $\IP$-twists associated to $\IP$-objects, see \cite{HT}. On the other hand, the study of spherical objects in families led to the notion of a spherical functor and its spherical twist, see, for example, \cite{ALdg}.
In \cite{Segal}, Segal gives a construction, how to realise any given autoequivalence as the twist of a spherical functor. But this construction gives a spherical functor whose source category will almost never be of geometric origin. This means that one cannot hope to find all autoequivalences of $D(X)$ only by looking out for spherical functors.

So it still is useful to look for other generalisations, like the $\IP$-twist associated to a $\IP$-functor.
A first (split) version was introduced independently by Addington in \cite{Add} and by Cautis in \cite{Cautis--flopsguide}.
Recently, this notion was generalised to (non-split) $\IP$-functors by Anno and Logvinenko in \cite{AL-P}.

We recall the essential part of this new definition of a $\IP^n$-functor.
Let $F \colon \mathcal A \to \mathcal B$ be an enhanced exact functor of enhanced triangulated categories (e.g. a Fourier-Mukai functor between two derived categories), which admits both adjoints $L, R \colon \mathcal B \to \mathcal A$.
In order to say that $F$ is a $\IP^n$-functor, we need that $RF$ is filtered by powers of an autoequivalence $H$ of $\mathcal A$. In categorical terms, this means that there is a Postnikov system of the form
\begin{equation}
\label{eq:postnikov}
\begin{tikzcd}
\id \ar[rr] && Q_1 \ar[rr] \ar[dl] && \cdots \ar[rr] \ar[dl] && Q_n \ar[ld]\\
& H \ar[dashed, ul] && H^2 \ar[dashed, ul] & \cdots& H^n \ar[dashed, ul]
\end{tikzcd}
\end{equation}
with $Q_i$ endofunctors of $\mathcal A$ and $Q_n \cong RF$.
There are more conditions to satisfy, see \cite[Def.~4.1]{AL-P}, but this part is all we need to know here.

There, Anno and Logvinenko ask whether their notion of a $\IP$-functor is symmetric, that is, whether adjoints of a $\IP$-functor are again $\IP$-functors.
This holds for spherical functors, but not for the split version of $\IP$-functors. Still, cyclic covers give rise to adjoint pairs of (genuinely non-split) $\IP$-functors, see \cite[\S 7.3]{AL-P}.
So it is natural to ask, whether this holds in general for their new notion.

We answer this question negatively, confirming the expectation stated in \cite{AL-P}.

\begin{maintheorem}
There are $\IP$-functors whose adjoints are not $\IP$-functors.
\end{maintheorem}

We show that there are examples of $\IP^n$-functors $F$, such that there is no possibility to realise a Postnikov system as in \autoref{eq:postnikov} for the adjoints $L,R$.
More precisely, we show that the adjoints are not $\IP^m$-functors for any $m\in\IN$ and any autoequivalence $H$ of $\mathcal B$ (so we rule out not only the case that $m=n$ and $H$ is the $\IP$-twist of $F$, which is what one might expect from the symmetry result for spherical functors; see \cite[Thm.~1.1]{ALdg}).

As a $\IP^n$-functor, we take the most simple there is, namely,
\[
F = P \otimes (\blank) \colon D(\kk) \to D(X)
\]
associated to a single $\IP^n$-object $P \in D(X)$.
For the counterexample, we choose a smooth projective $X$ of dimension $2n$ with big (anti-)canonical sheaf such that there is a $\IP^n$-object $P$ with  $\supp(P) \subsetneq X$.

The bigness gives us some control over the autoequivalences of $D(X)$, see \autoref{prop:bigauto}. This in turn allows us to rule out a Postnikov system as above for the adjoints of $F$ which is the main result \autoref{thm:main}.

In \autoref{sec:example}, we show that there are varieties satisfying the conditions above. There we start with a surface $S$ with big (anti-)canonical sheaf, such that there is a spherical object $E \in D(S)$ with one-dimensional support (for example, the second Hirzebruch surface). Then its Hilbert scheme of $n$ points has again a big (anti-)canonical sheaf and contains a $\IP^n$-object, to which \autoref{thm:main} applies.

\subsection*{Conventions and notations}

We denote by $\kk$ an algebraically closed field of characteristic zero.

For a variety $X$ over $\kk$, $D(X)$ denotes the bounded derived category of coherent sheaves on $X$, which is a triangulated category.
We write $\Hom^*(A,B)$ for the graded vector space $\bigoplus_{i\in\IZ} \Hom(A,B[i])[-i]$ of derived homomorphisms in $D(X)$.

All functors between derived categories are meant to be exact. In particular, we abuse notation and use $\otimes$  for the derived functor $\otimes^L$, that is, we use the same symbol as for the functor of coherent sheaves.

\section{Adjoints of $\IP$-functors are not necessarily $\IP$-functors}
\label{sec:main}

Before turning to the main result, we prove a statement about autoequivalences of $D(X)$, where $X$ is a smooth projective variety with big \mbox{(anti-)}canonical sheaf.
This is close in spirit to Kawamata's result \cite[Thm.~1.4]{KawamataDK} that D-equivalence implies K-equivalence for varieties of this type.

\begin{proposition}
\label{prop:bigauto}
Let $X$ be a smooth projective variety with either $\omega_X$ or $\omega_X^{-1}$ big. Then, for every $H\in \Aut(D(X))$, there is an $d\in \IZ$, non-empty open subsets $U_H,W_H\subset X$, and an isomorphism $\phi\colon U_H\xrightarrow\sim W_H$ with
\[
 H(\reg_x)=\reg_{\phi(x)}[d] \quad\text{for all $x\in U_H$.}
\]
\end{proposition}

\begin{proof}
The statement can be extracted from the proof of \cite[Thm.~1.4]{KawamataDK} as presented in \cite[\S 6]{Huy}.
Let us explain how exactly.

Let $\cP\in D(X\times X)$ be the Fourier--Mukai kernel of $H$, and write $\pr_i\colon X\times X\to X$ with $i=1,2$ for the projections to the factors.
By the proof of \cite[Prop.~6.19]{Huy}, there is an irreducible component $Z$ of $\supp(\cP)$ such that $\pr_{i\mid Z}\colon Z\to X$ is birational for both $i=1$ and $i=2$. (Explicitly, this is stated on top of page 149 of \emph{loc.\ cit.}\ for the normalisation $\widetilde Z$ in place of $Z$. But as $\widetilde Z\to Z$ is birational, it holds also for $Z$.)

Let $Z_1,\dots, Z_k$ be the further irreducible components of $\supp(\cP)$.
By \cite[Cor.~6.12]{Huy}, we have $\pr_1(Z_1), \dots,\pr_1(Z_k)\subsetneq X$. 
We define $U_H\subset X$ as the non-empty open subset which we optain by removing the $\pr_1(Z_j)$ for $j=1,\dots k$ as well as the images under $\pr_1$ of the exceptional loci of $\pr_{i\mid Z}$ for $i=1,2$. Then, for $x\in U_H$, the intersection $\supp(\cP)\cap(\{x\}\times X)$ consists of a single point. Hence $H(\reg_x)$ is supported on a single point $y\in X$. 
As $H(\reg_x)$ is a point-like object, we have $H(\reg_x)\cong \reg_y[d]$ for some $d\in \IZ$; see \cite[Lem.~4.5]{Huy}. Now, the assertion follows by \cite[Cor.~6.14]{Huy}.
\end{proof}

For the following theorem, we recall that an object $P$ in some enhanced triangulated category $\mathcal T$ is a \emph{$\IP^n$-object} (as introduced by Huybrechts and Thomas \cite{HT}) if 
\begin{itemize}
\item $\Hom^*(P,P) \cong \kk[t]/t^{n+1}$ as graded algebras with $\deg(t)=2$;
\item $\Hom^*(P,\blank) \cong \Hom^*(\blank,P[2n])^\vee$, that is, $P$ is a $2n$-Calabi--Yau object.
\end{itemize}
Such a $\IP^n$-object gives rise to a (split) $\IP^n$-functor $F = P \otimes (\blank) \colon D(\kk) \to \mathcal T$  in the sense of Addington \cite{Add} and Cautis \cite{Cautis--flopsguide}, and therefore $F$ is also a $\IP^n$-functor in the general sense of Anno and Logvinenko \cite{AL-P}.

\begin{theorem}
\label{thm:main}
Let $X$ be a $2n$-dimensional smooth projective variety with $n$ even and $\omega_X$ or $\omega_X^{-1}$ big.
Let $P\in D(X)$ be a $\IP^n$-object with $\supp(P)\subsetneq X$.
Then neither of the adjoints of the $\IP^n$-functor $F = P \otimes (\blank) \colon D(\kk) \to D(X)$ has the structure of a $\IP$-functor.
\end{theorem}

\begin{proof}
As the right and left adjoint of $F$ differ only by precomposition with the Serre functor $S_X$ (see \cite[Rem.~1.31]{Huy} and note that the Serre functor of $D(\kk)$ is the identity),
we can focus on one of them, say on the right adjoint $R = \Hom^*(P,\blank)$.

We assume the contrary, that is, that $R$ admits the structure of a $\IP^m$-functor for some $m$.
In particular, there is a Postnikov system of the following form
\[
\begin{tikzcd}
\id \ar[rr] && Q_1 \ar[rr] \ar[dl] && \cdots \ar[rr] \ar[dl] && Q_m \ar[ld]\\
& H \ar[dashed, ul] && H^2 \ar[dashed, ul] & \cdots& H^m \ar[dashed, ul]
\end{tikzcd}
\]
for some autoequivalence $H$ of $D(X)$. 
Note that we have an isomorphism $Q_m \cong R^R R$, where $R^R$ denotes the right adjoint of $R$.
We compute that $R^R = P \otimes \omega_X[2n] \otimes (\blank) \cong P[2n] \otimes (\blank)$,
so $Q_m \cong \Hom^*(P,\blank)\otimes P[2n]$.

Passing to the  Grothendieck group $K_0(X)$, the Postnikov system gives the equation
\begin{equation}
\label{eq:postnikov-K}
[Q_m (\blank)] = \sum_{i=0}^m [H^i(\blank)]\,.
\end{equation}

By \autoref{prop:bigauto}, there is a birational map $\phi\colon X\dashrightarrow X$, a non-empty open subset $U\subset X$, and some  $d\in \IZ$ with
\begin{equation}\label{eq:sky}
 H^i(\reg_x)\cong \reg_{\phi^i(x)}[di] \quad\text{for all $i=0,\dots m$ and all $x\in U$.}
\end{equation}
Concretely, start with $\phi\colon U_1\coloneqq U_H\to W_H\eqqcolon W_1$ as produced by \autoref{prop:bigauto}. Then define inductively $W_{i+1}\coloneqq W_i\cap U_i$ and $U_{i+1}\coloneqq \phi^{-1}(W_{i+1})$. Finally, set $U\coloneqq U_m$.

We now plug the skyscaper sheaf $\reg_x$ of some point $x\in U\setminus \supp(P)$ into \autoref{eq:postnikov-K}. We have
$Q_m(\reg_x) = 0$ as $\Hom^*(P,\reg_x)=0$ and
$[H^i(\reg_x)] = (-1)^{i \cdot d} [\reg_{\phi^i(x)}]$ by applying \autoref{eq:sky}.
Passing to cohomology via the Mukai vector, the images of all skyscraper sheaves of points become equal (and non-zero, namely a generator of $H^{4n}(X,\IC)$).
Hence \autoref{eq:postnikov-K} turns into
\[
0 = \sum_{i=0}^m (-1)^{i \cdot d} v(\reg_x) =
\begin{cases}
(m+1) v(\reg_x) & \text{if $d$ is even;}\\
v(\reg_x) & \text{if $d$ is odd and $m$ even;}\\
0 & \text{if $d$ and $m$ are odd.}
\end{cases}
\]
In particular, we conclude that $m$ has to be odd.

Next we plug our $\IP$-object $P$ into \autoref{eq:postnikov-K}.
For this we compute that $Q_m(P) =(\kk[t]/t^{n+1})\otimes P [2n]$ with $\deg(t)=2$, so we have
\[
A \coloneqq (n+1) [P] = \sum_{i=0}^m [H^i(P)] \eqqcolon B
\]
Computing the Euler pairing with itself on both sides gives:
\[
\begin{split}
\chi(A,A) &= (n+1)^3 \\
\chi(B,B) &= \sum_{i=0}^m \chi(H^i(P),H^i(P)) + \sum_{0\leq i \neq j \leq m} \chi(H^i(P),H^j(P))  \\
&= (m+1) \chi(P,P) +   2 \cdot \sum_{0\leq i < j \leq m} \chi(H^i(P),H^j(P))
\end{split}
\]
where we use that $H$ is an autoequivalence, hence commuting with the Serre functor, and that $P$ is a $2n$-Calabi--Yau object, in order to conclude that $\chi(H^i(P),H^j(P)) = \chi(H^j(P),H^i(P))$.
Note that $\chi(A,A)$ is necessarily odd (as $n$ is even by assumption), but $\chi(B,B)$ is even (as $m$ is already shown above to be odd), which is the desired contradiction.
\end{proof}

\begin{remark}
Let $F \colon \cA \to \cB$ be a functor with both adjoints.
We might call $F$ a \emph{$\IP^m$-like functor} if it admits a filtration of $RF$ by powers of an autoequivalence $H$ like in \autoref{eq:postnikov}, following the terminology introduced in \cite{HKP} and \cite{HM}.
What then turns a $\IP$-like functor into a $\IP$-functor are the additional conditions of \cite[Def.~4.1]{AL-P}.

So in the proof of \autoref{thm:main}, we showed more generally, that the adjoints of the $\IP^n$-functor $F = P \otimes (\blank)$ are not even $\IP$-like functors (so in particular not $\IP$-functors).
\end{remark}

\section{Examples}
\label{sec:example}

Let $S$ be a surface with big (anti-)canonical sheaf together with a spherical object $E\in D(S)$ such that $\supp(E)\subsetneq X$.

For example, the second Hirzebruch surface $\Sigma_2$ has a big anti-canonical sheaf (as any smooth, projective, toric variety) and contains a $(-2)$-curve $C$, that is, $C \cong \IP^1$ and $C^2=-2$.
In particular, $\reg_C$ is a spherical object in $D(\Sigma_2)$.
See, for example, \cite{AutoToric}.

For such a spherical object $E \in D(S)$, we find that $E^{\boxtimes n}\in D_{\sym_n}(S^n)$ is a $\IP^n$-object for every $n\in\IN$; see \cite[\S 4]{PS}. Under the derived McKay correspondence $D_{\sym_n}(S^n)\cong D(S^{[n]})$ this translates to a $\IP^n$-object $P\in D(S^{[n]})$ with $\supp(P)\subsetneq X\coloneqq S^{[n]}$; compare \cite[\S 6]{Formality}. Hence, together with the following lemma, we obtain examples that satisfy the hypotheses of \autoref{thm:main}.

\begin{lemma}
 Let $S$ be a smooth projective surface with big (anti-)canonical sheaf. Then, for every $n\in \IN$, also the (anti-)canonical sheaf of the Hilbert scheme $S^{[n]}$ of $n$ points on $S$ is big.
\end{lemma}

\begin{proof}
For an invertible sheaf $L$ on $S$, there is the associated invertible sheaf
\[
 L_{n}\coloneqq \mu^*(\pi_*^{\sym_n}L^{\boxtimes n})
\]
on $S^{[n]}$, where $\mu\colon S^{[n]}\to S^{(n)}\coloneqq S^n/\sym_n$ is the Hilbert--Chow morphism and  $\pi_*^{\sym_n}$ is taking invariants of the push-forward under the quotient morphism $\pi\colon S^n\to S^{(n)}$.

By the equivariant K\"unneth formula we find that
\begin{equation}\label{eq:globalsec}
\Ho^0(S^{[n]}, L_n)\cong \Ho^0(S^n, L^{\boxtimes n})^{\sym_n} \cong   \Symm_n \Ho^0(S,L)\,.
\end{equation}
Suppose that $L$ is a big invertible sheaf, this means that the growth of $\Ho^0(S,L^k)$ is of order $k^{2}$. Then \autoref{eq:globalsec} gives that the growth of $\Ho^0(S^{[n]},L_n^k)$ is of order $k^{2n}$, since
\[
\dim( \Symm_n (\kk^{ak^2}))=\binom{ak^2+n-1}{n}\ge \frac{a^n}{n!} k^{2n}\,.
\]
So $L_n$ is again big.

In particular, we obtain the statement as $\omega_{X^{[n]}}\cong (\omega_X)_n$ by \cite[Prop.~1.6]{NW}.
\end{proof}

\begin{remark}
If $E=\reg_C$ is the structure sheaf of a $(-2)$-curve $C\subset S$, then the associated $\IP^n$-object $P$ on $S^{[n]}$ is the structure sheaf of the subvariety $\IP^n\cong C^{[n]}\eqqcolon Z\subset S^{[n]}$; see \cite[Prop.~6.6]{Formality}. One can show for its normal bundle
$N_{Z/S^{[n]}}\cong \Omega_Z$, which means that $Z$ is the center of a Mukai flop.
Indeed, if $S$ is a K3 surface, then $S^{[n]}$ is holomorphic symplectic. On a holomorphic symplectic variety, the normal bundle of any embedded projective space of half the dimension of the ambient space is automatically isomorphic to the cotangent bundle; see \cite[Ex.~1.3(i)]{HT}. One can reduce to the case of a K3 surface since any two $(-2)$-curves on any two smooth surfaces have isomorphic analytic neighbourhoods, as follows from \cite[Satz 7]{Grauert}.

More generally, if $X$ is any $2n$-dimensional smooth projective variety with $n$ even and $\omega^{\pm 1}_X$ big together with a subvariety $Z\cong \IP^n$ with normal bundle $N_{Z/X}\cong \Omega_Z$ we have an example for \autoref{thm:main}; see \cite[Ex.~1.3(i)]{HT}.
\end{remark}

\end{document}